\documentclass[12pt,reqno]{article}

\usepackage[usenames]{color}
\usepackage{amssymb}
\usepackage{amsmath}
\usepackage{amsthm}
\usepackage{amsfonts}
\usepackage{amscd}
\usepackage{graphicx}

\usepackage{algorithm}
\usepackage{algpseudocode}
\algnewcommand{\algorithmicand}{\textbf{ and }}
\algnewcommand{\algorithmicor}{\textbf{ or }}
\algnewcommand{\OR}{\algorithmicor}
\algnewcommand{\AND}{\algorithmicand}
\algnewcommand{\var}{\texttt}
\algnewcommand\True{\textbf{true}\space}
\algnewcommand\False{\textbf{false}\space}

\usepackage[colorlinks=true,
linkcolor=webgreen,
filecolor=webbrown,
citecolor=webgreen]{hyperref}

\definecolor{webgreen}{rgb}{0,.5,0}
\definecolor{webbrown}{rgb}{.6,0,0}

\usepackage{color}
\usepackage{fullpage}
\usepackage{float}

\usepackage{graphics}

\usepackage{latexsym}
\usepackage{epsf}
\usepackage{breakurl}

\begin{document}
	
\theoremstyle{plain}
\newtheorem{theorem}{Theorem}
\newtheorem{corollary}[theorem]{Corollary}
\newtheorem{lemma}[theorem]{Lemma}
\newtheorem{proposition}[theorem]{Proposition}

\theoremstyle{definition}
\newtheorem{definition}[theorem]{Definition}
\newtheorem{example}[theorem]{Example}
\newtheorem{conjecture}[theorem]{Conjecture}

\theoremstyle{remark}
\newtheorem{remark}[theorem]{Remark}

\begin{center}
\vskip 1cm{\LARGE\bf Odd Spoof Multiperfect Numbers Of Higher Order} 
\vskip 1cm \large
L\'aszl\'o T\'oth\\
L-8476 Eischen \\
Grand Duchy of Luxembourg\\
\href{mailto:uk.laszlo.toth@gmail.com}{\tt uk.laszlo.toth@gmail.com}
\end{center}
\vskip .2 in

\begin{abstract}
We extend our previous work on odd spoof multiperfect numbers to the case where spoof factor multiplicities exceed $2$. This leads to the identification of $11$ new integers that would be odd multiperfect numbers if one of their prime factors had higher multiplicity. An example is $181545$, which would be an odd multiperfect number if only one of its prime
factors, $3$, had multiplicity $5$.
\end{abstract}
	
\section{Introduction}
Recall that $\sigma(n)$ denotes the sum-of-divisors function of the positive integer $n$, and $n$ is said to be \textit{perfect} if $\sigma(n)=2n$, and multiperfect (or $k$-perfect) if $\sigma(n) = kn$ for some positive integer $k\geq2$. No odd perfect numbers have been found so far, but Descartes observed that
$$
\mathcal{D}=198585576189
$$
would be an odd perfect number if only one of its composite factors, $22021$, were prime. Regrettably, $22021=19^2 \cdot 61$, so this is not the case. Since Descartes, much effort has been devoted to finding such ``spoof perfect'' numbers, without success. In our previous paper \cite{Toth25}, we presented a few numbers akin to $\mathcal{D}$, for instance
$$
S = 8999757 = 3^2 \cdot 13^2 \cdot 61 \cdot 97,
$$
which would be an odd multiperfect number if we assumed (wrongly) that one of its prime factors, $61$, were a square. Indeed, if that were the case, we would have
\begin{align*}
\sigma(S) &= (3^2 + 3 + 1) (13^2 + 13 + 1) (97 + 1) (61^2 + 61 + 1) \\
&= (13) \cdot (3\cdot61) \cdot (2 \cdot 7^2) \cdot ( 3 \cdot 13 \cdot 97) \\
&= 98 \cdot 3^2 \cdot 13^2 \cdot 61 \cdot 97 \\
&= 98 S.
\end{align*}
This led us to devise an algorithm to search for such numbers and found several more. In this paper, our aim is to develop our methods even further; first, by generalizing the concept of spoof $k$-perfect numbers, and second, by extending our search for numbers similar to $\mathcal{D}$ and $S$. As a result, we find $11$ new odd positive integers which would be multiperfect, if only one of their prime factors had higher multiplicity. One such example is
$$
T = 181545 = 3 \cdot 5 \cdot 7^2 \cdot 13 \cdot 19,
$$
which would be an odd multiperfect number, if only one of its prime factors, $3$ had multiplicity $5$:
\begin{align*}
\sigma(T) &= (3^5 + 3^4 + 3^3 + 3^2 + 3 + 1) (5 + 1) (7^2 + 7 + 1) \\
& \ \ \ \ \ (13 + 1) (19 + 1) \\
&= (2^2 \cdot 7 \cdot 13) (2 \cdot 3) (3 \cdot 19) (2 \cdot 7) (2^2 \cdot 5) \\
&= 192 \cdot 3 \cdot 5 \cdot 7^2 \cdot 13 \cdot 19 \\
&= 192 T.
\end{align*}
In the next sections, we will provide a generalization of spoof multiperfect numbers, and discuss some of their properties. We will then adapt Robin's classical inequality to spoof multiperfect numbers and provide details about the algorithm we used to find our results, including pseudo-code.

\section{Generalized spoof multiperfect numbers}
In our previous paper \cite{Toth25}, we defined two kinds of spoof multiperfect numbers. In particular, we designated the positive integer $s=nx$ as a spoof $k$-perfect number
\begin{enumerate}
    \item \textit{of the first kind} if $\sigma(n) (x+1) = knx$,
    \item \textit{of the second kind} if $\sigma(n) (x^2+x+1) = knx$,
\end{enumerate}
for a positive integer $k\geq2$. We shall now introduce an extension of this definition by allowing the spoof factor $x$ to have any multiplicity greater than $2$.
\begin{definition} [Spoof $k$-perfect number of order $\alpha$] \label{def-spoof} 
Let $s=nx$ be a positive integer such that $n,x \in \mathbb{N}$ and $n,x\geq2$. Furthermore, let $\alpha\geq1$ be an integer and define
$$
\mathcal{S}_\alpha = \sum_{a=0}^{\alpha} x^a. 
$$
Then, if
$$
\sigma(n) \mathcal{S}_\alpha = knx,
$$
for some positive integer $k$, then $s$ is a spoof $k$-perfect number \textit{of order} $\alpha$.
\end{definition}
Note that the case $\alpha = 1$ corresponds to the classical Descartes numbers, the case $\alpha=2$ to the numbers in our previous work (such as $8999757$), while the cases $\alpha>2$ form the basis of our study in this paper. A trivial example of an odd spoof $k$-perfect number of order $3$ is $s=15$. Indeed, if we assume (incorrectly) that its prime factor $3$ has multiplicity $3$, then we have
\begin{align*}
\sigma(s) &= (5 + 1) \cdot (3^3 + 3^2 + 3 + 1) \\
&= 2^4 \cdot 3 \cdot 5 \\
&= 16 s.
\end{align*}
In hopes of finding such numbers, we implemented an algorithm that finds all spoof multiperfect numbers of order $\alpha$ within a given range, which we outline in Section \ref{sec-algo} We were thus able to check all integers $s=nx$ with $n<1.6 \times 10^7$, of order $\alpha\leq10$. As a result, we found $14$ spoof multiperfect numbers, $11$ of which are new, for which $x$ is a prime that is also coprime to $n$. These integers are listed in Table \ref{tab-res}.
\begin{center}
\begin{table}[h!]
\centering
\caption{Odd spoof $k$-perfect numbers $s=nx$ of order $\alpha$}\label{tab-res}
\begin{tabular}{|c|c|c||c|c|}
\hline
$ s$  & $n$ & $x$ & $k$ & $\alpha$ \\ \hline
$15$  & $5$ & $3$ & $16$ & $3$ \\ \hline
$33$  & $11$ & $3$ & $44$ & $4$ \\ \hline
$1911$  & $637$ & $3$ & $152$ & $5$ \\ \hline
$1989$  & $153$ & $13$ & $280$ & $3$ \\ \hline
$34485$  & $11495$ & $3$ & $56$ & $4$ \\ \hline
$36309$  & $12103$ & $3$ & $160$ & $5$ \\ \hline
$77805$  & $11115$ & $7$ & $16$ & $2$ \\ \hline
$92781$  & $1521$ & $61$ & $97$ & $2$ \\ \hline
$105435$  & $21087$ & $5$ & $256$ & $4$ \\ \hline
$181545$  & $60515$ & $3$ & $192$ & $5$ \\ \hline
$241395$  & $80465$ & $3$ & $64$ & $4$ \\ \hline
$8999757$  & $147537$ & $61$ & $98$ & $2$ \\ \hline
$62998299$  & $1032759$ & $61$ & $112$ & $2$ \\ \hline
$440988093$  & $7229313$ & $61$ & $114$ & $2$ \\ \hline  
\end{tabular}
\end{table}
\end{center}
Note that the integers $s=77805, 92781$, and $8999757$ have already been discovered in our previous paper. We also note that the numbers $s=62998299$ and $440988093$ are remarkable because they also have $x=61$, which now accounts for the majority of odd spoof multiperfect numbers of order $2$. It also appears in Descartes' classical example $s=198585576189$, which is the only known odd spoof perfect number of order $1$.

Many other odd spoof multiperfect numbers exist, for which $x$ is either composite, or prime but not coprime to $n$. We have ommitted these numbers from the results that we share in this paper.

As we noted in our previous paper, one may notice at this point that multiperfect numbers of this magnitude should not appear so early, due to an inequality discovered by Guy Robin \cite{Ro84} in 1984, namely
$$
\sigma(n) < e^\gamma n \log \log n,
$$
where $\gamma$ is the Euler-Mascheroni constant and $n > 5040$, if and only if the Riemann Hypothesis holds true. It thus follows that we would expect a $k$-perfect number $n$ to appear only after
$$
n > e^{e^{k e^{-\gamma}}},
$$
which is not the case in the spoof examples above. This observation leads us to examine the ``spoof equivalent'' of this inequality, which we will do in the next section.

\section{Robin's inequality for spoof multiperfect numbers}
We begin by adapting Robin's inequality to spoof $k$-perfect numbers in the following manner.
\begin{lemma} \label{lemm-robin}
Let $s=nx$ denote a spoof $k$-perfect number of order $\alpha$. Furthermore, let $n>5040$. Then, assuming the Riemann Hypothesis, we have:
$$
\frac{kx}{\mathcal{S}_\alpha} < e^{\gamma} \log \log n,
$$
where
$$
\mathcal{S}_\alpha = \sum_{a=0}^{\alpha} x^a. 
$$
\end{lemma}
\begin{proof}
Let $s=nx$ denote a spoof $k$-perfect number of order $\alpha$. Thus, by Definition \ref{def-spoof}, we have
$$
\sigma(n) = \frac{knx}{\mathcal{S}_\alpha}.
$$
On the other hand, Robin's inequality states that for $n>5040$,
$$
\sigma(n) < e^{\gamma} n \log \log n.
$$
Combining these two above gives
$$
\frac{knx}{\mathcal{S}_\alpha} < e^{\gamma} n \log \log n,
$$
and after simplifying $n$ on both sides our claim is proved.
\end{proof}
A quick corollary of the above gives a bound on the components of the classical Descartes numbers.
\begin{corollary}
Let $s=nx$ denote a Descartes number with pseudo-prime factor $x$. Then, assuming the Riemann Hypothesis, we have:
$$
\frac{2x}{x+1} < e^{\gamma} \log \log n.
$$
\end{corollary}
\begin{proof}
This follows directly by applying Lemma \ref{lemm-robin} with $k=2$ and $\alpha=1$. Furthermore, we no longer need the restriction $n>5040$, since no Descartes numbers exist with $n$ smaller than $5040$.
\end{proof}

\section{Algorithm} \label{sec-algo}

In this small final section, we give a few details about the algorithm we used to find the results in this paper, which is very similar to the one in our previous work \cite{Toth25}. We run through positive integers $n$ and compute the quantity
$$
q = \frac{\sigma(n)}{kn}.
$$
Taking care that the fraction $q$ is in the lowest terms possible (i.e., the numerator $q_{num}$ and denominator $q_{den}$ have greatest common divisor $1$), we can compute their difference $\delta$:
$$
\delta = q_{den} - q_{num}.
$$
Then if 
$$
\delta = \sum_{a=0}^{\alpha} q_{num}^a - q_{num},
$$
we have found a spoof $k$-perfect number $s=nx$ of order $\alpha$, where the spoof factor is $x=q_{num}$.

In practical terms, we can check whether the positive integer $n$ is a suitable candidate as illustrated by the following pseudo-code.

\begin{algorithm}[H]
    \caption{Check whether a positive integer $n$ is an odd spoof $k$-perfect number of order $\alpha<\alpha_{max}$}\label{algo1}
    
    \begin{algorithmic}[0]

        \Procedure{CheckCandidate}{$n$, $\sigma_n$, $\alpha_{max}$, $k$}
        \State $q \gets \sigma_n /(k \times n)$
        \State \texttt{Reduce}$[q]$
        \State $num \gets $ \texttt{Numerator}$[q]$
        \State $den \gets $ \texttt{Denominator}$[q]$
        \State $delta \gets den - num$
        \For{$\alpha=1\xrightarrow{} \alpha_{max}$}
            \State $S_\alpha \gets $ \texttt{ComputeAlphaSum}[$n, \alpha$]

            \If{$delta == (S_{\alpha} - num) \AND num>1$}
                \State $s \gets n \times num$
                \If{\texttt{Mod}$[s,2] == 1$}
                    \State \texttt{Print}[''Found at '' + $s$]
                \EndIf
            \EndIf

        \EndFor
    \EndProcedure
    \end{algorithmic}
\end{algorithm}
Note that the call to \texttt{Reduce}$[q]$ ensures that the fraction $q$ is reduced to the lowest terms, as mentioned above. Furthermore, the \texttt{ComputeAlphaSum} function is a simple computer implementation of the sum we defined previously,
$$
\mathcal{S}_\alpha = \sum_{a=0}^{\alpha} x^a.
$$
Putting everything together, we iterate through our search space in the following manner.
\begin{algorithm}[H]
    \caption{Finding spoof multiperfect numbers of different orders, given the defined limits $n_{max}, k_{max}$, and $\alpha_{max}$}\label{algo2}
    \begin{algorithmic}[0]
    
        \Procedure{Main}{$n_{max}, k_{max}, \alpha_{max}$}
        
        \For{$n=1\xrightarrow{} n_{max}$}
            \State $\sigma_n \gets$\texttt{DivisorSigma}$[n]$

            \For{$k=2\xrightarrow{} k_{max}$}
                \State \texttt{CheckCandidate}$[n, \sigma_n, \alpha_{max}, k]$
            \EndFor
        \EndFor
        \\
        \Return $sum$
    \EndProcedure
    \end{algorithmic}
\end{algorithm}
Note that by computing $\sigma_n$ only once for each $n$, considerable computing time is gained, given that this operation is the most expensive one in the algorithm in terms of computing resources.

\section{Conclusion and further work}
In this paper we extended our previous work on odd spoof multiperfect numbers and found several new examples of odd positive integers that would be multiperfect, if only one of their prime factors had higher multiplicity. Since our algorithm is simple, it can easily be used to discover other examples with sufficient computing resources, and indeed we hope that the present work will encourage others to do so.

\section*{Acknowledgment}
The author extends his gratitude to several anonymous reviewers, whose comments and remarks have improved the quality of this paper.


\begin{thebibliography}{99}

\bibitem{Pace spoof group} N. Andersen, S. Durham, M. Griffin,  J. Hales, P. Jenkins, R. Keck, H. Ko, G. Molnar, E. Moss, P. Nielsen, K. Niendorf, V. Tombs, M. Warnick, and D. Wu, Odd, spoof perfect factorizations,\textit{ J. Number Theory} \textbf{234} (2022), 31--47.

\bibitem{Ro84} 
G. Robin, Grandes valeurs de la fonction somme des diviseurs et hypoth\`{e}se de Riemann, \textit{J. Math. Pures Appl.} {\bf 63} (1984), 187--213.

\bibitem{Toth21}
L. T\'oth, On the Density of Spoof Odd Perfect Numbers, \textit{Comput. Methods Sci. Technol.} \textbf{27 (1)} (2021), 25--28.

\bibitem{Toth25}
L. T\'oth, Odd Spoof Multiperfect Numbers, \textit{Integers} \textbf{25} (2025), Art. A19.

\end{thebibliography}
\end{document}